\documentclass[runningheads,envcountsame]{llncs}
\usepackage[12pt]{extsizes}
\usepackage{times}
\usepackage[margin=1in]{geometry}

\let\oldbibliography\thebibliography
\renewcommand{\thebibliography}[1]{\oldbibliography{#1}
\setlength{\itemsep}{0pt}}

\makeatletter
\renewcommand{\@Opargbegintheorem}[4]{%
  #4\trivlist\item[\hskip\labelsep{#3#2\@thmcounterend}]}
\makeatother

\spnewtheorem{conjecture}{Conjecture}{\bfseries}{\rmfamily}

\usepackage{amssymb,amsmath}

\usepackage{mathrsfs}
\usepackage{mathtools}
\usepackage{doi}
\DeclarePairedDelimiter\floor{\lfloor}{\rfloor}
\newcommand{\B}{\textbf}
\newcommand{\baro}{\overline}

 \usepackage{hyperref}
 \hypersetup{
     colorlinks=true,
     linkcolor=blue,
     filecolor=blue,
     citecolor = black,      
     urlcolor=blue,
     }

\begin{document}

\title{Cyclability, Connectivity and Circumference }
%
%
\author{Niranjan Balachandran\inst{1} \and
 Anish Hebbar\inst{2}}
%

%
\institute{Indian Institute of Technology Bombay, India  \and
 Indian Institute of Science, Bangalore, India\\
niranj@iitb.ac.in \quad  anishhebbar@iisc.ac.in 
}
\maketitle              
\begin{abstract}
In a graph $G$, a subset of vertices $S \subseteq V(G)$ is said to be cyclable if there is a cycle containing the vertices in some order. $G$ is said to be $k$-cyclable if any subset of $k \geq 2$ vertices is cyclable. If any $k$ \textit{ordered} vertices are present in a common cycle in that order, then the graph is said to be $k$-ordered. We show that when $k \leq \sqrt{n+3}$, $k$-cyclable graphs also have circumference $c(G) \geq 2k$, and that this is best possible. Furthermore when $k \leq \frac{3n}{4} -1$,  $c(G) \geq k+2$, and for $k$-ordered graphs we show $c(G) \geq \min\{n,2k\}$. We also generalize a result by Byer et al. \cite{Byer2007} on the maximum number of edges in nonhamiltonian $k$-connected graphs, and show that if $G$ is a $k$-connected graph of order $n \geq 2(k^2+k)$ with $|E(G)| > \binom{n-k}{2} + k^2$,  then the graph is hamiltonian, and moreover the extremal graphs are unique.

\end{abstract}
\section{Introduction}

We consider only finite, undirected, simple graphs throughout this paper. The vertex and edge sets of $G$ will be denoted by $V(G)$ and $E(G)$ respectively, the graph complement by $\baro{G}$. The length of the longest cycle in the graph $G$, also known as the circumference, will be denoted by $c(G)$. 
The minimum degree, independence number and connectivity of a graph will denoted by $\delta(G), \alpha(G)$ and  $\kappa(G)$ respectively. We will also use  $d_{H}(v)$ for the degree of $v$ in $H$. The set of neighbours of a vertex $v \in V(G)$ will denoted by $N(v)$, and the closed neighbourhood of $v$, viz. $N(v) \cup \{v\}$ will be denoted by $N[v]$. The join of two graphs $G_1, G_2$, denoted $G_1 \lor G_2$ is simply a copy of $G_1$ and $G_2$, with all edges between $V(G_1)$ and $V(G_2)$ also being present.\\ 


A subset $S \subseteq V(G)$ of vertices in a graph $G$ is said to be cyclable if $G$ has a cycle containing the vertices of $S$ in some order, possibly including other vertices. A graph $G$ is said to be $k$-cyclable if any $k \geq 2$ vertices of $G$ lie on a common cycle. Note that the problem of determining the hamiltonicity of a graph is a special case of cyclability, namely when $k = n$. Cyclability and connectivity are interlinked, as was shown by Dirac \cite{Dirac1960} who proved for every $k \geq 2$, $k$-connected graphs are also $k$-cyclable. In fact, for $k = 2$ connectivity and cyclability are equivalent, but in general for $k \geq 3$ it is not necessarily true that every $k$-cyclable graph is also $k$-connected, as can be seen by considering the graph $K_2 \lor 2K_k$ which has connectivity exactly $2$ and is also $k$-cyclable. For a brief survey of results involving conditions for cycles to contain a particular set, refer to \cite{Gould2009}.\\

There is a rich literature on conditions guaranteeing the presence of long cycles in graphs, the most classical one being that of Dirac \cite{Dirac1952} who showed that in $2$-connected graphs, the circumference is at least$c(G) \geq \min\{n, 2 \delta(G)\}$. Moreover, $k$-connected graphs have a circumference of at least $\min\{n,2k\}$ from an easy consequence of Menger's theorem, and this is tight. A famous result by Chvátal and Erd\H{o}s \cite{Chvatal1972} relates the connectivity and independence number of a graph to hamiltonicity, and says that if the connectivity of a graph $G$ is at least its independence number, then the graph is hamiltonian. However, not much is known when the requirement of connectivity is weakened to cyclability. Bauer et al. \cite{Bauer2000} obtained lower bounds for the length of the longest cycle in $3$-cyclable graphs in terms of the minimum degree and independence number, but not much else is known for $k$-cyclable graphs for arbitrary $k$.\\

 

Cyclability  has also received interest from an algorithmic and complexity theoretic point of view as it is a 'hard' parameter that can be thought of as a more quantitative measure of hamiltonicity. Since the classical HAMILTONIAN CYCLE problem is NP-complete, the problem of determining whether a graph is $k$-cyclable (CYCLABILITY) is NP-complete as well. 
The problem of determining whether a given subset $S$ of vertices is cyclable (TERMINAL CYCLABILITY) has been studied in the Parameterized Complexity framework (FPT) (parameterized by $|S|$) and the best known algorithm has  running time $O(2^{|S|} n^{O(1)})$ \cite{Bjorklund2012}. For some special classes of graphs such as interval graphs and bipartite permutation graphs, Crespelle and Golovach \cite{Crespelle2022} showed that both these problems can be solved in polynomial time. For $|S| = O((\log \log n)^{1/10})$, Kawarabayashi \cite{Kawarabayashi2008} obtained a polynomial time algorithm for TERMINAL CYCLABILITY.\\



Note that $k$-connectivity guarantees $c(G) \geq \min \{n,2k\}$ and also ensures $k$-cyclability. Thus, a natural question to ask is whether the same bound on the circumference can be obtained when the connectivity criteria is weakened to cyclability. When $k = n-1$, we would require any set of $n-1$ vertices of $G$ to lie on a common cycle. It turns out that in this case, it is not necessary that the graph is hamiltonian. Indeed, the existence of  hypohamiltonian graphs \cite{Doyen1975} of order $n$ is known for all $n \geq 18$. Our first  result in this paper gives a similar circumference bound for a wide range of $k$: 
\begin{theorem}\label{maintheorem}
Let $G$ be a $k$-cyclable graph, where $2 \leq k \leq n$. Then, 

\begin{equation*}
  c(G) \geq
    \begin{cases}
      2k & \text{if } k \leq \sqrt{n+3} \\
      k+2 & \text{if }  k \leq \frac{3n}{4} - 1
    \end{cases}       
\end{equation*}
Moreover, for $2 \leq k \leq \sqrt{n+3}$, this bound on the circumference is best possible.
\end{theorem}

\noindent  Note that for $k\ge\frac{n}{2}$ it is still possible that one can have a bound of the form $c(G)\ge (1+\gamma)k$ for some fixed positive constant $\gamma < 1$ as long as $k \neq n - o(n)$.\\

A related notion is the orderedness of a graph, a strong hamiltonian property that was first introduced by Ng
and Schultz \cite{Ng1997}.   A graph $G$ is said to be $k$-ordered if any sequence of distinct vertices $T = \{v_1, \ldots, v_k\}$ are present in some common cycle in that order, possibly including other vertices. Note that $k$-ordered graphs are naturally also $k$-cyclable, and it is also easy to see that they are $(k-1)$-connected. For a comprehensive survey of results on $k$-ordered graphs, see \cite{Faudree2001}. We show that for $k$-orderedness, the same circumference bound as $k$-connectivity holds for all $2 \leq k \leq n$.

\begin{theorem}\label{orderedtheorem}
Let $G$ be a $k$-ordered graph, $2 \leq k \leq n$. Then, $c(G) \geq \min\{n,2k\}$.
\end{theorem}

Our second pursuit in this paper is to obtain Tur\'an-type results for the circumference of $k$-connected graphs, specifically the maximum number of edges in nonhamiltonian $k$-connected graphs. A classical result states that if $G$ is a graph of order $n$ with $|E(G)| > \binom{n-1}{2} + 1$, then $G$ is hamiltonian. This was generalized by \cite{Byer2007} for $k \leq 3$, where they showed that if $G$ is $k$-connected and satisfies $|E(G)| > \binom{n-k}{2} + k^2$ with $n$ sufficiently large, then the graph is hamiltonian and the extremal graphs are unique. We further generalize their result and extend it to any $k$ satisfying $n \geq 2(k^2+k)$.

\begin{theorem} \label{kconhamtheorem}
Let $G$ be a $k$-connected graph of order $n \geq 2(k^2+k)$. If $|E(G)| > \binom{n-k}{2} + k^2$, then $G$ is hamiltonian. Moreover, the extremal graphs are unique.
\end{theorem}

The rest of the paper is organized as follows. We lay out some preliminaries in the next section, and give the proofs of Theorems \ref{maintheorem}, \ref{orderedtheorem}, and \ref{kconhamtheorem} in the following section. We conclude with some remarks and open questions.

\section{Preliminaries}

When the underlying graph is clear, we will use $\delta, \kappa, \alpha$ instead of $\delta(G), \kappa(G), \alpha(G)$ for brevity, and also omit the subscript in $d_H(v)$.
We also use the following well-known lemma attributed to Dirac repeatedly throughout the paper, and provide an outline of the proof for completeness.
\begin{lemma}[\cite{Dirac1960}] \label{kcontheorem}
Any $k$-connected graph $G$ is $k$-cyclable. Moreover, it satisfies $c(G) \geq \min\{n,2k\}$
\end{lemma}
\begin{proof}[Proof Sketch]
Suppose some subset $S$ of vertices with $|S| = k$ was not fully contained in any cycle. Then, take a cycle $C$ containing as many of the vertices of $S$ as possible, and pick some $v \in S$ that is not in $C$. By Menger's theorem, we can choose $k$ vertex-disjoint paths from $v$ to the cycle $C$, and these endpoints divide $C$ into $k$ segments. Since there are strictly less than $k$ vertices of $S$ in $C$, one of the segments does not contain any vertex from $S$, and thus we can extend this segment with the $2$ disjoint paths from $v$ at the ends of the segment to obtain a cycle containing more vertices of $S$, contradiction. \\
Now consider the longest cycle $C$ in $G$ and suppose its length is strictly less than $\min\{n,2k\}$. Pick some $v \in V(G)$ not in $C$, and by Menger's theorem there are $k$ vertex disjoint paths from $v$ to $C$. By the pigeonhole principle, some two endpoints of these $k$ paths must be adjacent on the cycle $C$, giving a contradiction as we can replace the edge between these endpoints with the $2$ paths to obtain a longer cycle. \qed
\end{proof}

A famous result by  Chvátal and Erd\H{o}s  states the following

\begin{theorem}[\cite{Chvatal1972}] \label{erdostheorem}
If in a graph $G$, $\alpha(G) \leq \kappa(G)$, then $G$ is hamiltonian.
\end{theorem}

A natural generalization of the above is to flip the condition $\alpha(G) \leq \kappa(G)$, and instead ask for lower bounds on the circumference of a graph $G$ where $\alpha(G) \geq \kappa(G)$. Foquet and Jolivet \cite{J.L.Fouquet} conjectured the following, which was later proven by Suil O, Douglas B. West and Hehui Wu.

\begin{theorem}[\cite{O2011}]\label{westtheorem}
If $G$ is a $k$-connected $n$-vertex graph with independence number $\alpha$ and $\alpha \geq k$, then $G$ has a cycle of length at least $\frac{k(n+k-\alpha)}{\alpha}$.
\end{theorem}

The following result by Dirac is well-known and was a precursor to a number of results involving the length of the longest cycle in a graph.

\begin{theorem}[\cite{Dirac1952}]\label{diractheorem}
If $G$ is $2$-connected and has minimum degree $\delta$, $c(G) \geq \min\{2 \delta,n\}$.
\end{theorem}

Note that $2$-connectivity is equivalent to $2$-cyclability. Bauer et al. obtained a bound on the circumference of $3$-cyclable graphs in terms of the minimum degree and independence number.

\begin{theorem}[\cite{Bauer2000}]\label{bauertheorem}
 If $G$ is $3$ cyclable, then $$c(G) \geq min\{n, 3\delta - 3 , n + \delta - \alpha\}.$$
\end{theorem}

Ng and Schultz studied a related hamiltonian property termed $k$-orderedness, and showed the following connectivity result. Once again, we include the proof for completeness.

\begin{lemma}[\cite{Ng1997}]\label{schultztheorem}
Let $G$ be a $k$-ordered graph. Then, $G$ is $(k-1)$-connected.
\end{lemma}
\begin{proof}
If not, there exists a set $S$ of $k-2$ vertices whose removal disconnects $G$, breaking it into at least $2$ components. Take $2$ vertices $u,v$ in different components, then any path from $u$ to $v$ must go through some vertex of $S$. Thus, let $T$ consist of $u$, $v$ and then the vertices of $S$, in that order. These vertices must appear in some cycle in that order, giving a contradiction. \qed

\end{proof}

We will also need the concept of graph closure introduced by Bondy and Chvátal. Define
the closure of $G$, denoted $cl(G)$, to be the graph obtained by repeatedly joining any two nonadjacent vertices $x,y$ that satisfy $d(x) + d(y) \geq n$ in $G$. They showed that $cl(G)$ is well-defined (independent of the order in which nonadjacent vertex pairs are considered), and that $G$ is hamiltonian if and only if $cl(G)$ is also hamiltonian.

\begin{lemma}[\cite{Bondy1976}]\label{bondytheorem}
Suppose $cl(G) = G$ for a nonhamiltonian graph $G$ of order $n$. Then $d(x) + d(y) \leq n-1$ for any pair $\{x,y\}$ of  nonadjacent vertices.
\end{lemma}

This was later generalized to obtain results for higher order connectivity, the bounds now also involving the independence number. We define $$\sigma_{k}(G) = \min \{\sum_{i=1}^{k} d(x_i), \{x_1, \ldots x_{k} \} \text{ an independent set of size k in G}\} $$

Note that $\sigma_1(G)$ simply corresponds to the minimum degree $\delta$, and Ore's theorem \cite{Ore1960} states that if $\sigma_2(G) \geq n$, then the graph is hamiltonian.

\begin{theorem}[\cite{Li2013}] \label{haotheorem}
Let $G$ be a $k$-connected graph of order $n$ and independence number $\alpha$. If $\sigma_{k+1}(G) \geq n+ (k-1) \alpha - (k-1)$, then $G$ is hamiltonian.
\end{theorem}

\section{Proofs of the Results}

\begin{proof}[Proof of Theorem \ref{maintheorem}]
$ $\newline
We will first prove the bound for the regime $2 \leq k \leq \sqrt{n+3}$. \\Consider any $k$-cyclable graph with $\alpha(G) \geq k$. Then, let $S$ be a set of $k$ independent vertices, and consider the cycle containing it. This gives us a cycle of length at least $2k$, as any $2$ independent vertices are not adjacent to each other. Thus, we can assume $\alpha(G) \leq  k - 1$. Let the connectivity of the graph be $\kappa$. Using Theorem \ref{westtheorem}, it suffices to show

$$ \frac{\kappa(n+\kappa - \alpha)}{\alpha} \geq 2k   \iff n \geq 2k(\frac{\alpha}{\kappa}) + (\alpha - \kappa)$$

 \noindent As $k$-cyclable graphs are also $2$-cyclable, and thus $2$-connected, we must have $\kappa \geq 2$. Hence, it is sufficient to show the stronger inequality

$$ n \geq 2k (\frac{k-1}{\kappa}) + k - 3$$ which is always true when $$n \geq k^2-3 \iff k \leq \sqrt{n+3} $$
Note that if we only ask for an improvement of the form $c(G) \geq (1 + \gamma)k$ for some positive constant $\gamma < 1$, we can improve the range of $k$ for which the result holds. Once again, let $S$ be any set of at least $ \frac{(1+ \gamma)k}{2}$ many independent vertices, and consider the cycle containing $S$. This corresponds to a cycle containing at least $(1+ \gamma)k$ many vertices since any two independent vertices are not adjacent, and thus we get $\alpha < \frac{(1+ \gamma)k}{2}$. Similar to the previous argument, if the connectivity of the graph is $\kappa$, by Theorem \ref{westtheorem} it suffices to show

$$ \frac{\kappa(n+\kappa - \alpha)}{\alpha} \geq (1 +\gamma) k   \iff n \geq  (1 +\gamma) k  (\frac{\alpha}{\kappa}) + (\alpha - \kappa)$$

\noindent Using $\kappa \leq 2$ and  $\alpha < \frac{(1+ \gamma)k}{2}$, we are done as long as

$$ n \geq \frac{(1+\gamma)^2k^2}{4} + \frac{(1+\gamma)k}{2} - 2 \iff \frac{\sqrt{4n+9}}{1+ \gamma} \geq k$$
So the above argument only yields a linear improvement in $c(G)$ for $k$ up to around $2 \sqrt{n}$.\\

Now, suppose $2 \leq k \leq \frac{3n}{4} - 1$, and assume to the contrary that $c(G) < k+2$. We must have $k \geq 3$ as $2$-cyclable graphs are $2$-connected and hence have circumference at least $4$ for $n \geq 4$. By Theorem
 \ref{diractheorem}, we must have $\delta \leq \frac{k+1}{2}$. Moreover, $\alpha \leq \frac{k+1}{2}$ as otherwise we could simply take a cycle containing $\alpha$ many independent vertices. Consider a vertex $v$ with minimum degree $\delta$, with neighbourhood $N(v)$ satisfying $|N(v)| = \delta$. Now, choose $v$ and any $k-1$ vertices from $V  \backslash N[v]$, which is possible as long as $k-1 \leq n - 1 - \delta$. Then, any cycle containing these vertices must also contain some $2$ neighbours of $v$, giving $c(G) \geq k+2$, and we are done.

Thus, we must have $k + \delta > n $. Note that when $2 \leq k \leq \frac{3n}{4} -1$, $n \geq k+2$ if $n \geq 4$. So, we must either have $3\delta - 3 \leq k+1$ or $n + \delta - \alpha \leq k+1$, otherwise we are done by Theorem \ref{bauertheorem}.\\  The former inequality gives $\delta \leq \frac{k+4}{3}$, which gives $$n < k + \delta \leq \frac{4k+4}{3} \implies \frac{3n-4}{4} < k$$ a contradiction. Hence, we must have $\delta \geq \frac{k+5}{3}$, $\alpha \leq \frac{k+1}{2}$ giving $$ k+1 \geq n + \delta - \alpha \geq n + \frac{k+5}{3} - \frac{k+1}{2} = n + \frac{7-k}{6}$$ 

\noindent or equivalently,  $\frac{3n}{4} -1 \geq k \geq \frac{6n+1}{7}$, which is again a contradiction. \qed
\end{proof}

\noindent We now prove an analogous bound for the circumference of $k$-ordered graphs.

\begin{proof}[Proof of Theorem \ref{orderedtheorem}]
$ $\newline
We know that $k$-ordered graphs are also $k-1$ connected from Theorem \ref{schultztheorem}, thus $\kappa \geq k-1$. We also must have $\alpha \leq k-1$, as otherwise we can simply take $k$ independent vertices in any order to obtain a cycle of size at least $2k$, in which case we are done. Hence, $$\kappa \geq k-1 \geq \alpha$$ so by Theorem \ref{erdostheorem}, we have that $G$ is hamiltonian, and thus we are done in this case as well. \qed
\end{proof}

In fact, it is not hard to see that the $\min \{n, 2k\}$ bound  on the circumference is achieved for all $2 \leq k \leq n$. If $k > n/2$, simply consider the complete graph $K_n$ which is clearly $k$-connected, $k$-ordered, $k$-cyclable and has circumference $n$. If $k \leq n/2$, consider the complete bipartite graph $G = K_{k,n-k} = (A,B,E)$, which is $k$-ordered, and hence $k$-cyclable.  Indeed, take any sequence of $k$ distinct vertices $T = (v_1, v_2, \ldots, v_k)$. We construct a cycle containing $T$ in that order as follows.\\

Let $T_A$ be the set of vertices in $T$ and $A$, with $T_B$ being defined similarly. Then, for any $v \in T_A$, if the next vertex in the sequence $T$ is in $T_B$, then simply follow the edge joining them. Otherwise, first follow an edge to a vertex in $B \backslash T_B$, and then back to the next vertex which must have been in $T_A$. Follow the same procedure for vertices in $T_B$. At the end, follow the edge joining the first and last vertex. We cannot run out of vertices as the number of extra vertices outside $T_A$ in $A$ that are needed is at most $|T_B|$, and $|A|= k = |T_A| + |T_B|$. Similarly, $|B| = n -k \geq k =|T_A| + |T_B|$.\\

We now generalize a result by \cite{Byer2007} on the maximal number of edges in a $k$-connected nonhamiltonian graph, for $k=2,3$. We will need the following short lemma which appears in \cite{Byer2007}.

\begin{lemma}[\cite{Byer2007}] \label{byertheorem}
Let $G$ be a nonhamiltonian, $k$-connected graph of order $n$. Then $k \leq \frac{n-1}{2}$ and $|E(\baro{G})| \geq \binom{k+1}{2} + (k-1)(n-k-1) - \sigma_{k+1}(G)$
\end{lemma}
\begin{proof}
By Theorem \ref{erdostheorem}, $k$-connected nonhamiltonian graphs must contain an independent set $I = \{x_1, \ldots , x_{k+1} \}$ of $k+1$ vertices. The graph is disconnected on removal of the the $n -(k+1) $ vertices of $G - I$, thus we must have
$ n - (k+1) > k-1$, or $k \leq \frac{n-1}{2}$.

Now consider the independent set $I$ satisfying $\sum_{i=1}^{k+1} d(x_i) = \sigma_{k+1}(G)$. Let the edges in $\baro{G}$ incident on at least one vertex of $I$ be denoted $X_I$. Then $X_I$ contains $\binom{k+1}{2}$ edges with both endpoints in $I$ and $\sum_{i=1}^{k+1} (n-1 - k -  d_G(x_i))$ edges with exactly one endpoint in $I$. Thus, we obtain
$$ \, \, \quad \quad  \quad\quad \quad \quad \quad \quad  \quad\quad |E(\baro{G})| \geq |X_I| = \binom{k+1}{2} +   (k-1)(n-k-1) - \sigma_{k+1}(G) \quad \quad \quad  \quad  \quad    \, \, \, \, \, \qed $$  
\end{proof}
Using a slight variation of the above result and Lemma \ref{bondytheorem}, \cite{Byer2007} also show the following result.

\begin{lemma}[\cite{Byer2007}] \label{byertheorem2}
Suppose $G = cl(G)$ for a nonhamiltonian graph $G$ of order $n$, and $m \leq \alpha(G)$. Then

\begin{equation*}
  |E(\baro{G})| \geq
    \begin{cases}
      \frac{m}{2}(n-m) & \text{for n odd}\\
      \frac{m}{2}(n-m) + \frac{m}{2} - 1 & \text{for n even}
    \end{cases}       
\end{equation*}
\end{lemma}

\noindent With the above results, we are ready to proceed to the proof of Theorem \ref{kconhamtheorem}. The idea is that if $n$ is not that much bigger than $\alpha$, then we can get a sufficient lower bound on $|E(\baro{G})| $ using Lemma \ref{byertheorem2}. Otherwise, $n$ is much bigger than $\alpha$, and we can use Theorem \ref{haotheorem} and Lemma \ref{byertheorem}. To show the uniqueness of the extremal graphs, we will make use of the fact that these graphs must satisfy Lemma \ref{bondytheorem} \textit{maximally}, i.e., addition of any further edge causes a violation of the condition.

\begin{proof}[Proof of Theorem \ref{kconhamtheorem}]
$ $\newline \
First of all, assume $k \geq 2$ as we already know that when $|E(G)| > \binom{n-1}{2} + 1$, then $G$ is hamiltonian and consequently connected as well.
Assume $G$ is nonhamiltonian. We may assume $G = cl(G)$, in which case $d(x)+ d(y) \leq n-1$ for any two nonadjacent vertices $x,y$, from Lemma \ref{bondytheorem}. It suffices to prove that $$|E(\baro{G})| \geq \binom{n}{2} - \left( \binom{n-k}{2} + k^2\right) = k\cdot n - \frac{3k^2+k}{2} $$ Note first that if $\sigma_{k+1}(G) \leq n + k^2- k -1$, by Lemma \ref{byertheorem}
$$|E(\baro{G})| \geq \binom{k+1}{2} + (k+1)(n-k-1) - (n+k^2 - k - 1) = k \cdot n -  \frac{3k^2+k}{2} $$
as desired. We now assume $\sigma_{k+1}(G) \geq n + k^2 - k$ and show that in this case, $|E(\baro{G})|$ is \textit{strictly} greater than $k\cdot n - \frac{3k^2+k}{2}$. We will divide the problem into two cases, depending on the size of $n$ compared to $\alpha$.\\

\noindent \B{Case 1:} Assume $n > \frac{ (k^2-1) \cdot \alpha  + y}{k}$, where $y = \frac{-k^3 + 4k^2 + 3k + 2}{2}$.

Let $I = \{x_1,x_2,\ldots ,x_{k+1}\}$ be a set of $k+1$ independent vertices satisfying $\sum_{i=1}^{k+1}d(x_i) = \sigma_{k+1}(G)$, and assume without loss of generality that $$d(x_1) \geq \frac{\sigma_{k+1}(G)}{k+1} \geq \frac{n+k^2-k}{k+1}$$ 

\B{Subcase 1a:} Suppose $d(x_1) \geq n - 2k$. Note that $V(G) - I - N(x_1)$ is non-empty, as otherwise we would have $d(x_1) = n - k -1$, giving $d(x_i) \leq k$ for $2 \leq i \leq k+1$ as $d(x_1) + d(x_i) \leq n-1$ for $2 \leq i \leq k +1$.  This contradicts $\sigma_{k+1}(G) \geq n + k^2 - k$.  Thus, pick some $v \in V(G) - I - N(x_1)$, giving $d_{\baro{G}}(v) = n - 1 - d_G(v) \geq d_G(x_1) \geq n - 2k$. Therefore, $\baro{G}$ contains at least $n - 2k - |I| = n - 3k - 1$ edges with both endpoints not in $I$. Using the same bound we got in Lemma \ref{byertheorem} but also including the extra edges in $\baro{G}$ incident with $v$ (that have no endpoint in $I$) and using Theorem \ref{haotheorem}, we obtain
\begin{align*}
|E(\baro{G})| &\geq \binom{k+1}{2} + (k+1)(n-k-1) + (n - 3k - 1) - \sigma_{k+1}(G)\\ 
& \geq (k+2)\cdot n - \frac{k^2+9k+4}{2} - (n + (k-1) \alpha - k)\\
& > k\cdot n -  \frac{3k^2+ k}{2} +  \frac{3k^2+ k}{2} -  \frac{k^2+9k+4}{2} + k +  \frac{ (k^2-1) \cdot \alpha  + y}{k} - (k-1)\alpha\\
& =  (k\cdot n -  \frac{3k^2+ k}{2}) +  \frac{  (k-1)\cdot \alpha  + y + k(k^2-3k-2)}{k}  >  (k\cdot n -  \frac{3k^2+ k}{2})
\end{align*}
as desired, where the last inequality follows from $y = \frac{-k^3 + 4k^2 + 3k + 2}{2}$.\\

\B{Subcase 1b:} Suppose next that $d(x_1) \leq n - 2k- 1$. Then there exist distinct vertices $v_1, v_2 \ldots, v_{k} \in V(G) - I - N(x_1)$, and $\baro{G}$ contains at least $$(d_{\baro{G}}(v_1) - k - 1) + (d_{\baro{G}}(v_2) - k - 2) + \cdots + (d_{\baro{G}}(v_k) - 2k) = \sum_{i=1}^k d_{\baro{G}}(v_i) - \frac{3k^2 + k}{2}$$ edges with neither endpoint in $I$. Using $d(v_i) + d(x_1) \leq n-1$ as $G = cl(G)$, we get $d_{\baro{G}}(v_i) \geq d_{G}(x_1) \geq \frac{n + k^2 -k}{k+1}$ for all $1 \leq i \leq k$. Consequently, we obtain at least $$\frac{k(n+k^2-k)}{k+1} - \frac{3k^2+k}{2}$$ edges in $\baro{G}$ with neither endpoint in $I$. Using Theorem \ref{haotheorem} and Lemma \ref{byertheorem} again, we get

\begin{align*}
|E(\baro{G})| &\geq \binom{k+1}{2} + (k+1)(n-k-1) + \frac{k(n+k^2-k)}{k+1} - \frac{3k^2+k}{2} - (n + (k-1)\alpha - k)\\
& = (kn  - \frac{3k^2+k}{2}) + \frac{k}{k+1}n -(k-1)\alpha +  \binom{k+1}{2} -(k+1)^2 + \frac{k(k^2-k)}{k+1} + k\\
&> (kn  - \frac{3k^2+k}{2})  + \frac{k}{k+1} \frac{(k^2-1)\alpha + y}{k} - (k-1)\alpha + \frac{-k^2-k-2}{2} + \frac{k(k^2-k)}{k+1}  \\
& = (kn  - \frac{3k^2+k}{2}) + \frac{1}{k+1} \left( \frac{-k^3+4k^2+3k+2}{2}+ \frac{(-k^2-k-2)(k+1)}{2}  + k^3 - k^2\right) \\
& = kn - \frac{3k^2+k}{2}
\end{align*} as desired.\\

\noindent  \B{Case 2:} Assume $n \leq \frac{(k^2-1)\alpha + y}{k}$.\\
In this case, $\alpha \geq \frac{nk - y}{k^2-1}$. By Lemma \ref{byertheorem2}, $|E(\baro{G})| \geq \frac{1}{2} \alpha(n- \alpha)$. This is a upward facing parabola for fixed $n$, so for $\frac{nk - y}{k^2-1} \leq \alpha \leq n - \frac{nk - y}{k^2-1}$, this function is minimized at $\alpha = \frac{nk - y}{k^2-1}$. Therefore, in this range

\begin{align*}
|E(\baro{G})| &\geq \frac{\alpha}{2} (n - \alpha) \geq \frac{1}{2} (\frac{nk - y}{k^2-1})(\frac{n(k^2-k-1) +y}{k^2-1})\\
& = \frac{ n^2k(k^2 - k - 1) + n(2k+1-k^2)y - y^2)}{2(k^2-1)^2} 
\end{align*}

\noindent If we want the above to be strictly greater than $kn - \frac{3k^2+k}{2} $, $$\frac{ n^2k(k^2 - k - 1)}{2(k^2-1)^2}  \geq kn \iff n \geq \frac{2(k^2-1)^2}{k^2 - k - 1} = 2(k^2 + k + \frac{1-k}{k^2-k-1})$$ suffices. This is because for $k \geq 5$, $y = \frac{-k^3 + 4k^2 3k+2}{2} < 0$ and $2k+1 - k^2 < 0$, giving $(2k+1-k^2)(y) > 0$. Similarly, $ -y^2 = \frac{(-k^3 + 4k^2 + 3k+2)^2}{4} > -(3k^2+k)(k^2-1)^2$ for $k \geq 5$, so we only have to check the cases of $k = 2,3,4$ manually which is a routine check.\\

Now, it remains to consider the possiblility that $\alpha > n - \frac{nk - y}{k^2-1} = \frac{ n(k^2-k-1) +y}{k^2-1}$. In this case however,  $\alpha$ is quite large compared to $n$, so the $\binom{\alpha}{2}$ edges in $\baro{G}$ between the vertices of an independent set of size $\alpha$ is strictly greater than $ k \cdot n - \frac{3k^2 + k}{2}$ for all $n$. Indeed, we manually verify for $k \leq 3$, and for $k \geq 4$ simply note that $\frac{nk}{2} + y \geq 0$, and hence when $n \geq 2(k^2+k)$ we have $$\alpha > \frac{n(k^2 - \frac{3k}{2} -1) }{k^2 -1} \geq \frac{9n}{15}, \quad \binom{ 9n/15}{2} > \frac{9n}{30} \cdot \frac{8n}{15} > kn$$

We now prove that the extremal  nonhamiltonian $k$-connected graphs  are unique for $n \geq 2(k^2+k)$, by making use of Lemma \ref{bondytheorem}. Recall that we may assume  $G = cl(G)$ is a
nonhamiltonian, $k$-connected graph of order $n \geq 2k^2 + 2k$ with $\sigma_{k+1}(G)=n+k^2 - k - 1$ as equality only holds if all the inequalities in the above proof are tight.\\

Thus, all the edges in $\baro{G}$ have atleast one endpoint in $I$. Let $I =\{x_1, x_2, \ldots, x_{k+1}\}$ be a set of independent vertices such that $k \leq d(x_1) \leq \ldots \leq d(x_{k+1}) $. Note that $k$-connected graphs have minimum degree at least $k$ as otherwise, the graph could be disconnected by removing at most $k-1$ vertices.  As mentioned in the previous section, we may further assume that all edges in $\baro{G}$ have at least one endpoint in $I$, that is, if $x,y \in V(G) - I$, then $\{x,y\} \in E(G)$. We will now use the properties of graph closure repeatedly. First, note that we must have a clique on the remaining $n - k - 1$ vertices, each of which has degree at least $n - k -2$.
\begin{itemize}
    \item Say $d(x_k) \geq k + 1$  Consider the neighbours of $x_k$ in the clique. These neighbours have degree at least $n - k - 1$, and hence since $G = cl(G)$, must be adjacent to $x_{k+1}$ as well as $d(x_{k+1}) \geq k+1$, But then, these neighbours have degree at least $n - k$, and hence must be adjacent to all of $x_1, \ldots, x_{k+1}$ by the same argument. 
    Thus, $I$ and $N(I)$ together form a complete bipartite graph with $|N(I)| \geq k+1 = |I|$. If $d(x_{k+1}) > k+1$, then it is easy to see that the graph is hamiltonian, and otherwise $k+1 = d(x_i) \, \forall i \in [k+1]$, giving $$ \sigma_{k+1} = n+ k^2 - k -1 =  (k+1)^2 \iff n = 3k+2$$ which is false as we assumed $n \geq 2k^2 + 2k$.
    
    \item Otherwise $d(x_k) = k$, 
    , and hence $d(x_{k+1}) = \sigma_{k+1} - k^2 = n - k - 1$, so we have a clique on the $n- k$ vertices in $G \backslash  \{x_1, \ldots, x_{k}\}$.  The neighbours of any $x_i, i \in [k]$ must have degree at least $n-k$, and hence are joined to all the $x_i$. Thus, we obtain the desired extremal graph with exactly $\binom{n-k}{2} + k^2$ many edges, namely a clique on $n-k$ vertices and $k$ other independent vertices forming a complete bipartite graph with some $k$ vertices from the clique. \qed
\end{itemize} 
\end{proof}

\section{Concluding Remarks}

A simpler proof of Theorem \ref{maintheorem} with a weaker constant can be obtained using Tur\'an's theorem and a theorem of Erd\H{o}s and Gallai \cite{Erdos1959} on the length of the longest cycle in a graph. 
Consider any $k$-cyclable graph with $\alpha(G) \geq k$. Then, let $S$ be a set of $k$ independent vertices, and consider the cycle containing it. This gives us a cycle of length atleast $2k$, as any two independent vertices are not adjacent to each other. Thus, we must have $\alpha(G) < k$. By a variant of Tur\'an's theorem, we also have $\alpha > \frac{n}{\tilde{d} + 1}$, where $\tilde{d}$ is the average degree. Thus, we obtain

$$\frac{2|E(G)|}{n} + 1 = \tilde{d} + 1 > \frac{n}{\alpha} \geq \frac{n}{k-1} \implies |E(G)| \geq \frac{1}{2} n \left( \frac{n}{k-1} - 1 \right)$$

which is larger than $\frac{1}{2} (2k-1)(n-1)$ if $n \geq 2 k^2$. giving $c(G) \geq 2k$ when $k \leq \sqrt{n/2}$.\\

It is also interesting to understand what happens to the circumference of $k$-cyclable graphs for large values of $k$. As mentioned earlier in the introduction, it is not necessarily the case that $c(G) = n$ when $k = n-1$ due to the existence of hypohamiltonian graphs. Thus, we have the following extremal problem. 

\begin{conjecture}
For a given $n$, let $f(n)$ be the largest value of $k$ such that any $k$-cyclable graph satisfies $c(G)>k$. From the above, we have $f(n) < n-1$ and from Theorem \ref{maintheorem}, we know $f(n) = \Omega(n)$. Is it the case that $f(n) = n-2 ?$
\end{conjecture}

\noindent We can also ask for what regime of $k$ as a function of $n$ do results of the type in Theorem \ref{maintheorem} hold.

\begin{conjecture}
For a given $n$, let $g(n)$ be the largest value of $k$ such that any $k$-cyclable graph satisfies $c(G)\geq 2k$. From Theorem \ref{maintheorem} we know $g(n) = \Omega(\sqrt{n})$. Is it the case that $g(n) = O(\sqrt{n}) ?$
\end{conjecture}

Moreover, our results only give an improvement of the form $c(G) \geq (1+ \gamma) k$, $0< \gamma < 1$, for $k$ up to around $2\sqrt{n}$, and it is natural to ask if such a linear bound on the circumference can be obtained for much larger regimes of $k$. Finally, note that the results of Theorem \ref{kconhamtheorem} only hold for $n \geq 2(k^2 + k)$. For fixed values of $k \leq 3$, \cite{Byer2007} give a tight bound for the minimum value of $n$ for this to hold. They also note that this bound cannot hold for $k = \Omega(n)$,  in particular if $p = \floor{\frac{n-1}{2}}$, the graph obtained by joining $n-p$ independent vertices to each vertex of $K_p$ is $k$-connected and nonhamiltonian, with total number of edges more than $\binom{n-k}{2} + k^2$ when $\frac{n+1}{6} < k  < \floor{\frac{n-1}{2}}$. This still leaves a significant gap in the possible range of $k$ for which $k$-connectivity and $|E(G)| > \binom{n-k}{2} + k^2$ implies hamiltonicity, as our result only applies for $k = O(\sqrt{n})$.


\
%
%
%


\bibliographystyle{splncs04}
 \bibliography{caldam}

\begin{thebibliography}{10}
\providecommand{\url}[1]{\texttt{#1}}
\providecommand{\urlprefix}{URL }
\providecommand{\doi}[1]{https://doi.org/#1}

\bibitem{Bauer2000}
Bauer, D., McGuire, L., Trommel, H., Veldman, H.J.: {Long cycles in 3-cyclable
  graphs}. Discrete Mathematics  \textbf{218}(1-3), ~1--8 (2000).
  \doi{10.1016/S0012-365X(99)00331-3}

\bibitem{Bjorklund2012}
Bj{\"{o}}rklund, A., Husfeldt, T., Taslaman, N.: {Shortest cycle through
  specified elements}. Proceedings of the Annual ACM-SIAM Symposium on Discrete
  Algorithms pp. 1747--1753 (2012). \doi{10.1137/1.9781611973099.139}

\bibitem{Bondy1976}
Bondy, J.A., Chvatal, V.: {A method in graph theory}. Discrete Mathematics
  \textbf{15}(2),  111--135 (1976). \doi{10.1016/0012-365X(76)90078-9}

\bibitem{Byer2007}
Byer, O.D., Smeltzer, D.L.: {Edge bounds in nonhamiltonian k-connected graphs}.
  Discrete Mathematics  \textbf{307}(13),  1572--1579 (2007).
  \doi{10.1016/j.disc.2006.09.008}

\bibitem{Chvatal1972}
Chv{\'{a}}tal, V., Erd{\"{o}}s, P.: {A note on Hamiltonian circuits}. Discrete
  Mathematics  \textbf{2}(2),  111--113 (1972).
  \doi{10.1016/0012-365X(72)90079-9}

\bibitem{Crespelle2022}
Crespelle, C., Golovach, P.A.: {Cyclability in graph classes}. Discrete Applied
  Mathematics  \textbf{313},  147--178 (2022). \doi{10.1016/j.dam.2022.01.021}

\bibitem{Dirac1952}
Dirac, G.A.: {Some theorems on abstract graphs}. Proceedings of the London
  Mathematical Society  \textbf{s3-2}(1),  69--81 (1952).
  \doi{10.1112/plms/s3-2.1.69}

\bibitem{Dirac1960}
Dirac, G.A.: {In abstrakten Graphen vorhandene vollst{\"{a}}ndige 4‐Graphen
  und ihre Unterteilungen}. Mathematische Nachrichten  \textbf{22}(1-2),
  61--85 (1960). \doi{10.1002/mana.19600220107}

\bibitem{Doyen1975}
Doyen, J., {Van Diest}, V.: {New families of hypohamiltonian graphs}. Discrete
  Mathematics  \textbf{13}(3),  225--236 (1975).
  \doi{10.1016/0012-365X(75)90020-5}

\bibitem{Erdos1959}
Erdős, P., Gallai, T.: {On maximal paths and circuits of graphs}. Acta
  Mathematica Academiae Scientiarum Hungaricae  \textbf{10}(3-4),  337--356
  (1959). \doi{10.1007/BF02024498}

\bibitem{Faudree2001}
Faudree, R.J.: {Survey of results on k-ordered graphs}. Discrete Mathematics
  \textbf{229}(1-3),  73--87 (2001). \doi{10.1016/S0012-365X(00)00202-8}

\bibitem{Gould2009}
Gould, R.J.: {A look at cycles containing specified elements of a graph}.
  Discrete Mathematics  \textbf{309}(21),  6299--6311 (2009).
  \doi{10.1016/j.disc.2008.04.017}

\bibitem{J.L.Fouquet}
{J.L. Fouquet}, J.J.: {Probl{\'{e}}me 438}. Probl{\'{e}}mes combinatoires et
  th{\'{e}}orie des graphes, Univ. Orsay, Orsay, 1976

\bibitem{Kawarabayashi2008}
Kawarabayashi, K.I.: {An improved algorithm for finding cycles through
  elements}. Lecture Notes in Computer Science (including subseries Lecture
  Notes in Artificial Intelligence and Lecture Notes in Bioinformatics)
  \textbf{5035 LNCS},  374--384 (2008). \doi{10.1007/978-3-540-68891-4{\_}26}

\bibitem{Li2013}
Li, H.: {Generalizations of Dirac's theorem in Hamiltonian graph theory-A
  survey}. Discrete Mathematics  \textbf{313}(19),  2034--2053 (2013).
  \doi{10.1016/j.disc.2012.11.025}

\bibitem{Ng1997}
Ng, L., Schultz, M.: {k-Ordered Hamiltonian Graphs}. Journal of Graph Theory
  \textbf{24}(1),  45--57 (1997).
  \doi{10.1002/(SICI)1097-0118(199701)24:1<45::AID-JGT6>3.0.CO;2-J}

\bibitem{O2011}
O, S., {West Douglas B.}, D.B., Wu, H.: {Longest cycles in k-connected graphs
  with given independence number}. Journal of Combinatorial Theory. Series B
  \textbf{101}(6),  480--485 (2011). \doi{10.1016/j.jctb.2011.02.005}

\bibitem{Ore1960}
Ore, O.: {Note on Hamilton Circuits}. The American Mathematical Monthly
  \textbf{67}(1), ~55 (Jan 1960). \doi{10.2307/2308928}

\end{thebibliography}

\end{document}